\documentclass[12 pt]{amsart}
\usepackage{amssymb}
\usepackage{amsmath}
\usepackage{amscd}
\usepackage{graphicx}
\usepackage{latexsym}

\calclayout\evensidemargin .6in

\def\R{\mathbb{R}}

\def\D{\mathbb{D}}
\def\S{\mathbb{S}}

\newtheorem{theorem}{Theorem}[section]
\newtheorem{lemma}[theorem]{Lemma}

\newtheorem{corollary}[theorem]{Corollary}

\theoremstyle{definition}

\newtheorem{remark}[theorem]{Remark}
\newtheorem{definition}[theorem]{Definition}

\title[Contact Symplectic Fibrations and Fiber Connected Sum]
{Contact Symplectic Fibrations and \\ Fiber Connected Sum}

\author{Mehmet Firat Arikan}
\address{Dept. of Mathematics, Middle East Technical University, Ankara, TURKEY}
\email{farikan@metu.edu.tr}
\thanks{The author is partially supported by TUBITAK grant 1109B321200181 and also by NSF FRG grant DMS-1065910}


\begin{document}
\begin{abstract}
We consider certain type of fiber bundles with odd dimensional compact contact base, exact symplectic fibers, and the structure group contained in the group of exact symplectomorphisms of the fiber. We call such fibrations ``contact symplectic fibrations''. By a result of Hajduk-Walczak, some of these admit contact structures which are ``compatible'' (in a certain sense) with the corresponding fibration structures. We show that this result can be extended to get a compatible contact structure on any contact symplectic fibration, and also that isotopic contact structures on the base produce isotopic contact structures on the total space. Moreover, we prove that the fiber connected summing of two contact symplectic fibrations along their fibers results in another contact symplectic fibration which admits a compatible contact structure agreeing with the original ones away from the region where we take fiber connected sum, and whose base is the contact connected sum of the original contact bases.

\end{abstract}

\maketitle

%==============================================================================
%
%                 INTRODUCTION
%
%==============================================================================
\section{Introduction}

In what follows, all manifolds, maps and fiber bundle structures will be assumed to be differentiable. Suppose that $E$ is any locally trivial fiber bundle such that the fibers are of even dimension but the base space $B$ could have dimension in either parity. Let $\pi:E \twoheadrightarrow B$ be the fibration map and $G \subset \textrm{Diff}(F)$ the structure group. Let us denote such a fiber bundle by the symbol
\begin{equation} \label{eqn:fiber_bundle}
(F\hookrightarrow E \stackrel{\pi}{\twoheadrightarrow} B)_G.
\end{equation}
If there exist a symplectic structure $\omega$ on the fiber $F$ and all transition maps are symplectomorphisms of $(F,\omega)$ (i.e., $G \subset \textrm{Symp}(F,\omega)$), then the fiber bundle in (\ref{eqn:fiber_bundle}) is called a \emph{symplectic fibration}. Note that, for a symplectic fibration, each fiber $F_b:=\pi^{-1}(b)\approx F$ is equipped with a symplectic structure $\omega_p$ obtained from pulling back $\omega$ using a local trivialization around $b \in B$, and this construction is independent of the choosen trivialization as transition maps lies in $\textrm{Symp}(F,\omega)$. \\

In order to motivate the reader, let's recall some related work in even dimensions. Given a symplectic fibration as above suppose that the total space $E$ admits a symplectic form $\Omega$. (In particular, $B$ is assumed to be even dimensional.) Then $\Omega$ is said to be \emph{compatible with} the fibration $\pi:E \twoheadrightarrow B$ if it restricts to a symplectic form on every fiber $F_b$. Note that the compatibility condition implies that $(F_b,\omega_b)$ symplectically embeds into $(E,\Omega)$ for each $b\in B$. This definition is motivated by the result which states that on a locally trivial fiber bundle as in (\ref{eqn:fiber_bundle}) with a connected base if $E$ admits a symplectic form $\Omega$ which restricts to a symplectic form on every fiber, then $\pi:E \twoheadrightarrow B$ is a symplectic fibration compatible with $\Omega$. Conversely, due to a result of Thurston, for a symplectic fibration with a connected symplectic base if the symplectic form $\omega_b$ represents a cohomology class in $H^2(E)$, then one can construct a symplectic form on $E$ which is compatible with $\pi$. For more details about symplectic topology and symplectic fibrations with even dimensional base, we refer the reader \cite{MS}.\\

In this note, we consider similar approaches to build ``compatible'' contact structures on the total spaces of exact symplectic fibrations (see below for definitions).  In particular, the base spaces will be odd dimensional and admit contact structures. Here we only consider co-oriented contact structures, that is, those which appear as kernels of certain globally defined $1$-forms: A \emph{positive contact form} on a smooth oriented $(2n+1)$-dimensional manifold $M$
is a $1$-form $\alpha$ such that $\alpha \wedge (d\alpha)^n > 0$ (i.e., $\alpha \wedge (d\alpha)^n$ is a positive volume form on $M$). The hyperplane field (of rank $2n$) $\xi=\textrm{Ker} (\alpha)$ of a contact form $\alpha$ is called a \emph{positive (co-oriented) contact structure} on $M$. The pair $(M,\xi)$ (or sometimes $(M,\alpha)$) is called a \emph{contact manifold}. We say that two contact manifolds $(M,\xi)$ and $(M',\xi')$ are \emph{contactomorphic} if there exists a
diffeomorphism $f:M\longrightarrow M'$ such that $f_\ast(\xi)=\xi'$. Two contact structures $\xi_0, \xi_1$ on a $M$ are said to be \emph{isotopic} if there exists a $1$-parameter family $\xi_t$ ($0\leq t\leq 1$) of contact structures joining them (Gray's Stability). Note that isotopic contact structures give contactomorphic contact manifolds by Gray's Theorem. More details about contact topology can be found in \cite{Ge}.\\

Through out the paper, we will assume the fibers of any fiber bundle to be noncompact or compact with nonempty boundaries, (and, therefore, the symplectic structures on them will be exact). In Section \ref{sec:Contact_Symplectic_fibrations}, we will show the existence of a compatible contact structure on a fiber bundle with exact symplectic fibers, compact contact base and structure group contained in the group of exact symplectomorphisms of the fiber. One should note that Hajduk-Walczak \cite{HW} showed the existence of contact structures on some of such fiber bundles which are covered by the corresponding theorem of the present paper. Therefore, one can say that the related results of this note generalize the corresponding statements in \cite{HW}. After constructing contact structures on total spaces of bundles, in Section \ref{sec:Bundle_Contact_Structures} we also discuss how they are affected if we vary the contact structures on the bases in their isotopy classes.\\

Fiber sum operation on fiber bundles or more generally fiber connected summing of manifolds along their submanifolds is a well known technique to obtain new fiber bundles and manifolds out of given ones. The results from \cite{Gompf} and \cite{MW} show that under suitable conditions, fiber sum process can be done in the symplectic category. Geiges \cite{Geiges} proved that an analog result holds also in contact category under certain conditions. In Section \ref{sec:Fiber_connected_sum}, by following the fiber connected sum description given in \cite{Geiges} we will see that under suitable assumptions one can take the fiber connected sum of two fiber bundles (equipped with compatible contact structures) along their fibers which results in another fiber bundle whose total space admiting a compatible contact structure which agrees with the original ones away from the region where we take fiber connected sum, and whose base is the contact connected sum of the original contact bases.

%==================================================================
\medskip \noindent {\em Acknowledgments.\/} The author would like to
thank T\"UB\.ITAK, The Scientific and Technological Research Council of Turkey, for supporting this research, and also Y\i ld\i ray Ozan for helpful conversations.
%==================================================================

%==============================================================================
%
%         Contact Symplectic Fibrations
%
%==============================================================================

\section{Contact Symplectic Fibrations} \label{sec:Contact_Symplectic_fibrations}

A contact manifold $(Y,\textrm{Ker}(\beta))$ is called \emph{strongly symplectically filled} by a symplectic manifold $(F,\omega)$ if there exists a Liouville vector field $\chi$ of $\omega$ defined (at least) locally near $\partial F=Y$ such that $\chi$ is transversally pointing out from $Y$ and $\iota_{\chi}\omega=\beta$ on $Y$. In such a case, we also say that $(Y,\textrm{Ker} (\beta))$ is the \emph{convex boundary} of $(F,\omega)$. \\

An \emph{exact symplectic manifold} is a noncompact manifold $F$ (or a compact manifold with boundary), together with a symplectic form $\omega$ and a $1$-form $\beta$ satisfying $\omega=d\beta$. If $F$ is a compact manifold with boundary, then we also require that the $\omega$-dual vector field $\chi$ of $\beta$ defined by $\iota_\chi \omega=\beta$ should point strictly outwards along $\partial F$, in particular, $\beta|_{\partial F}$ is a contact form which makes $\partial F$ the convex boundary of $F$. Since $\beta$ determines both $\omega$ and $\chi$, it suffices to write an exact symplectic manifold as a pair $(F,\beta)$. The $1$-form $\beta$ is called an \emph{exact symplectic structure} on $F$. Since any symplectic form determines an orientation, any exact symplectic manifold $(F,\beta)$ is (symplectically) oriented. An \emph{exact symplectomorphism} $\varphi:(F_1,\beta_1)\to (F_2,\beta_2)$ between two exact symplectic manifolds  is a diffeomorphism such that $\varphi^*(\beta_2)-\beta_1$ is exact. \\%The \emph{completion} of $(F,\beta)$ is obtained from $F$ by gluing the positive part $\partial F \times [0,\infty)$ of the symplectization $(\partial F \times \mathbb{R}, d(e^t \beta|_{\partial F}))$ of its convex boundary. \\

Let us start with defining the analog of the above compatibility notion (given in the introduction) in odd dimensions.

\begin{definition} \label{def:compatibility}
Let $(F\hookrightarrow E \stackrel{\pi}{\twoheadrightarrow} B)_G$ be a fiber bundle with odd dimensional base such that the total space $E$ admits a contact structure $\xi$. Suppose that the fiber space $F$ is a noncompact manifold or a compact manifold with boundary. Then $\xi$ is said to be \emph{compatible with} $\pi$ if there exists a contact form $\alpha$ for $\xi$ such that $\alpha$ restricts to an exact symplectic structure on every fiber of $\pi$.
\end{definition}

Given an exact symplectic manifold $(F,\beta)$, one can define the group $$\textrm{Exact}(F,\beta):=\{\phi \in \textrm{Diff}(F) \,|\, \phi^*(\beta)-\beta \textrm{ exact}\}.$$ of exact symplectomorphisms which is a subgroup of the symplectic group $$\textrm{Symp}(F,d\beta)=\{\phi \in \textrm{Diff}(F) \,|\, \phi^*(d\beta)=d\beta\}.$$ Also if $\partial F \neq \emptyset$, one can consider the subgroup $$\textrm{Exact}(F,\partial F,\beta) \subset \textrm{Exact}(F,\beta)$$ of exact symplectomorphisms which are identity near $\partial F$.\\

We can now formally introduce our main objects of interest:

\begin{definition} \label{def:contact_symplectic_fibration}
A \emph{contact symplectic fibration} is a fiber bundle $$(F\hookrightarrow E \stackrel{\pi}{\twoheadrightarrow} B)_G$$
with exact symplectic fiber $(F,\beta)$ and the structure of group of exact symplectomorphisms (i.e., $G \subset \textrm{Exact}(F,\beta)$) such that the base space $B$ is compact and admits a co-oriented contact structure.
\end{definition}

From the definition one might understand that the phrase ``contact symplectic fibration'' refers to a fiber bundle with a compact contact base and an exact symplectic fibers, and also that the adjective ``contact'' in the phrase only emphasizes that the base space is contact. However, the main result of this section  (Theorem \ref{thm:Existence_on_contact_symp_fib}) shows that the word ``contact'' has, indeed, a more global meaning. \\

Given a fiber bundle $(F\hookrightarrow E \stackrel{\pi}{\twoheadrightarrow} B)_G$, if the base and/or fiber spaces are manifolds with boundaries, then the boundary of $E$ consists of two parts: The \emph{vertical boundary component} $\partial_v E:=\pi^{-1}(\partial B)$, and the \emph{horizontal boundary component} $\partial_h E:=\bigcup_{b \in B} \partial E_b$ where $E_b=\pi^{-1}(b)\cong F$ is the fiber over $b \in B$.

\begin{theorem} \label{thm:Existence_on_contact_symp_fib}
The total space of a contact symplectic fibration $(F\hookrightarrow E \stackrel{\pi}{\twoheadrightarrow} B)_G$ admits a (co-oriented) contact structure $\emph{Ker}(\sigma)$ which is compatible with $\pi$. Moreover, If $F$ has a boundary and $G \subset \emph{Exact}(F,\partial F,\beta)$, then the contact form $\sigma$ is equal to a product form on some collar neighborhood of $\partial_hE \cong B \times \partial F$.
\end{theorem}

Note that this theorem can be regarded as the analog of Thurston's theorem mentioned above in odd dimensions. Before proving the theorem, let us see some examples of contact symplectic fibrations appeared under a different name in the literature. To this end, we need to recall exact symplectic fibrations from \cite{S1,S2}. Here for our purposes we define them in a slightly different way.

\begin{definition} \label{def:Exact_Symplectic_Fibration}
An \emph{exact symplectic fibration} is a fiber bundle $$(F\hookrightarrow (E,\lambda) \stackrel{\pi}{\twoheadrightarrow} B)_G$$ equipped with a $1$-form $\lambda$ on $E$ such that each fiber $E_b$ with $\lambda_b=\lambda|_{E_b}$ is an exact symplectic manifold. An exact symplectic fibration is said to be \emph{trivial near horizontal boundary} if $\partial F \neq \emptyset$ and the following triviality condition near $\partial_h E$ is satisfied: Choose a point $b \in B$ and consider the trivial fibration $\tilde{\pi} : \tilde{E}:= B \times E_b \to B$ with the form $\tilde{\lambda}$ which is the pullback of $\lambda_b$, respectively. Then there should be a fiber-preserving diffeomorphism $\Upsilon:N \to \tilde{N}$ between neighborhoods $N$ of $\partial_h E$ in $E$ and $\tilde{N}$ of $\partial_h \tilde{E}$ in $\tilde{E}$ which maps $\partial_h E$ to $\partial_h \tilde{E}$, equals the identity on $N \cap E_b$, and $$\Upsilon ^* \tilde{\lambda}=\lambda.$$
\end{definition}

Lemma 1.1 of \cite{S2} implies that for an exact symplectic fibration $$(F\hookrightarrow (E,\lambda) \stackrel{\pi}{\twoheadrightarrow} B)_G,$$ the structure group $G$ falls into $\textrm{Exact}(F,\beta)$ where $\beta=\lambda_b$ is the restriction of the $1$-form $\lambda$ on $E$ to any fixed fiber $F=E_b$. The same result also implies that if we have an exact symplectic fibration which is trivial near horizontal boundary, then its structure group is contained in $\textrm{Exact}(F,\partial F,\beta)$. Also as discussed after Lemma 1.1, using the fact that $d\lambda_b$ is nondegenerate on $E_b$ for any $b \in B$, one can define a preferred (or compatible) connection on any exact symplectic fibration. Moreover, one can proceed in the other direction as well, that is, any fiber bundle with structure group contained in $\textrm{Exact}(F,\beta)$ (resp. $\textrm{Exact}(F,\partial F,\beta)$) and equipped with a compatible connection admits a structure of exact symplectic fibration (resp. a structure of exact symplectic fibration which is trivial near horizontal boundary). Putting these observations together for exact symplectic fibrations we obtain

\begin{lemma} \label{lem:Bundle_Structure}
Any exact symplectic fibration $(F\hookrightarrow (E,\lambda) \stackrel{\pi}{\twoheadrightarrow} B)_G$ is equipped with a fiber bundle structure with transition maps contained in $\emph{Exact}(F,\beta)$, that is, $G \subset \emph{Exact}(F,\beta)$, where $\beta=\lambda|_F=\lambda_b$ is the restriction of $\lambda$ to any fixed fiber $F=E_b$. If it is trivial near horizontal boundary, then $G \subset \emph{Exact}(F,\partial F,\beta)$. \qed
\end{lemma}

As an immediate consequence of this lemma, we have

\begin{corollary}
Any exact symplectic fibration with a compact base which admits a co-oriented contact structure can be equipped with a structure of a contact symplectic fibration. \qed
\end{corollary}

For completeness let us provide the following elementary fact which will be used in the proof of Theorem \ref{thm:Existence_on_contact_symp_fib}:

\begin{lemma} \label{lem:Contact_Str_on_M^Odd_cross_Exact}
The product of a contact manifold with an exact symplectic manifold admits a contact structure.
\end{lemma}

\begin{proof}
Let $(B^{2n+1},\mu)$ be any contact manifold, and $(F^{2m},\beta)$ an exact symplectic manifold. Consider the product $B \times F$ and the standard projections given by $$\pi_1:B \times F \longrightarrow B, \qquad \pi_2:B \times F \longrightarrow F.$$
Consider the pull-back forms $\tilde{\mu}:=\pi_1^{\ast}(\mu)$, $\tilde{\beta}:=\pi_2^{\ast}(\beta)$ and set
$$\sigma:=\tilde{\mu}+\tilde{\beta}.$$
Then by the binomial expansion formula and using the facts $(d\tilde{\mu})^i=0, \; \forall i\geq n+1$ and $(d\tilde{\beta})^j=0, \; \forall j\geq m+1$, we compute $$\sigma \wedge (d\sigma)^{n+m}=
\left (\begin{array} c n+m \\ m \end{array} \right)
\, \tilde{\mu} \wedge (d\tilde{\mu})^n \wedge (d\tilde{\beta})^m.$$ Therefore, $\sigma \wedge (d\sigma)^{n+m}$ is a volume form on the product manifold $B \times F$ because $\mu \wedge (d\mu)^n$ and $(d\beta)^m$ are volume forms on $B$ and $F$, respectively. Hence, $\sigma$ is a contact form on $B \times F$.
\end{proof}

\begin{remark}
After the above observations, one can prove Theorem \ref{thm:Existence_on_contact_symp_fib} in the following way: Given a contact symplectic fibration as in Definition \ref{def:contact_symplectic_fibration}, as the transition maps are contained in $\textrm{Exact}(F,\beta)$ we can glue the forms $\beta$ on each locally trivial piece of the fibration $\pi:E \to B$ to get a smooth $1$-form, say $\lambda$, on $E$ which restricts to an exact symplectic structure on each fiber of $\pi$. Then by following the lines in the proof of Lemma \ref{lem:Contact_Str_on_M^Odd_cross_Exact}, it is easy to check that the pull-back form $\mu+\lambda$ on $E$ defines contact structure which is compatible with $\pi$. For the purpose of the subsequent parts of the paper we now prove Theorem \ref{thm:Existence_on_contact_symp_fib} in a different perspective.
\end{remark}

\begin{proof}[Proof of Theorem \ref{thm:Existence_on_contact_symp_fib}]
Let $(F\hookrightarrow E \stackrel{\pi}{\twoheadrightarrow} B)_G$ be any  given contact symplectic fibration with $\textrm{dim}(F)=2m$ and $\textrm{dim}(B)=2n+1$. By the fiber bundle structure, we know that up to diffeomorphism the total space $E$ is obtained by patching the trivial $F$-bundles, and the transition maps are contained in the structure group $G \subset \textrm{Exact}(F,\beta)$ where $(F,\beta)$ is any fixed fiber of $\pi$. More precisely, there is an open cover $\{U_\alpha\}_\alpha$ of $B$ and a collection of diffeomorphisms $\phi_\alpha: \pi^{-1}(U_\alpha) \to U_\alpha \times F$ such that the diagram
$$\begin{array}{ccc}
\pi^{-1}(U_\alpha)&\stackrel{\phi_\alpha}{\longrightarrow}&U_\alpha \times F\\
\;\; \downarrow \pi&& \quad \;\;\downarrow \small{pr_1}\\
U_\alpha &=& U_\alpha
\end{array}$$
commutes. Moreover, by restricting $\phi_\alpha$ to any fiber $F_b=\pi^{-1}(b)$ and then projecting onto $F$-factor, we get a smooth map $\phi_\alpha(b):F_b \to F$, and the transition maps $\phi_{\gamma \alpha}: U_\alpha \cap U_\gamma \to G \subset \textrm{Exact}(F,\beta)$ given by $$\phi_{\gamma \alpha}(b)=\phi_\gamma(b) \circ \phi_\alpha(b)^{-1}.$$

Using paracompactness, one can find another open cover $\{V_\rho\}_\rho$ of $B$ such that for any index $\rho$ the closure $\overline{V}_\rho$ is contained in $U_\alpha$ for some $\alpha$. Since $B$ is compact (by assumption), there is a finite subcover $\{V_1,V_2, ..., V_r\}$ of $\{V_\rho\}_\rho$. Here one can assume that $V_i$'s are listed in a specific order so that $V_{i} \cap V_{i+1} \neq \emptyset$ whenever $V_i$ and $V_{i+1}$ belong to the same connected component of $B$. Indeed, the argument below can be done independently for each connected component. Therefore, without loss of generality $B$ will be assumed to be connected. Now consider the collection $\mathcal{W}:=\{W^i \subset B\, |\, i=1,...,r \}$ of compact sets defined by $$W^1=\overline{V}_1, \quad W^i=\overline{(\overline{V}_1 \cup \cdots \cup \overline{V}_{i}) -W^{i-1}}, \quad i=2,3, ..., r.$$
Clearly, $\mathcal{W}$ forms a partition for the base space $B$, that is, we have
\begin{equation} \label{eqn:Partition}
B=W^1 \cup W^2 \cup \cdots \cup W^r,
\end{equation}
for any $i,j$, $\textrm{int} (W^i) \cap \textrm{int} (W^j) =\emptyset$, and each union $W^1 \cup W^2 \cup \cdots \cup W^i$ intersects with $W^{i+1}$ along a smooth compact hypersurface $$H_i:=(W^1 \cup W^2 \cup \cdots \cup W^i) \cap W^{i+1} \subset B.$$ Observe that for each $i=1,...,r-1$ we have the map  $\phi_{i}: H_i \to G \subset \textrm{Exact}(F,\beta)$ obtained by restricting the corresponding $\phi_{\gamma \alpha}$. More precisely, from the construction of $\{V_\rho\}_\rho$, we know for each $i$ and $j$ with $1\leq j \leq i$ that $W^j \subset U_{\alpha_j}$ and $W^{i+1} \subset U_\gamma$ for some $\alpha_j,\gamma$. So one can define the map $\phi_i$ by patching the restrictions of $\phi_{\gamma \alpha_j}$ to the subsets $W^j \cap W^{i+1} \subset U_{\alpha_j} \cap U_\gamma$ together. As a result, up to diffeomorphism we obtain the description given by
\begin{equation} \label{eqn:Description}
E=(W^1 \times F) \cup_{\Phi^1} (W^2 \times F ) \cup_{\Phi^2} \cdots \cup_{\Phi^{r-1}} (W^r \times F )
\end{equation}
where each $\Phi^i:H_i \times F \to H_i \times F$ is the smooth gluing map defined by $$\Phi^i(b,p)=(b,\phi_i(b)(p)).$$
For each $i$ the fact that $\phi_i \in \textrm{Exact}(F,\beta)$ implies that there exists a smooth function $\psi_i:F \to \R$ such that $\phi_i^*(\beta)=\beta + d \psi_i$. In accordance with the smooth gluing map (or rule) $\Phi^i$, one can patch all these functions together which yields a smooth function $\Psi_i:H_i \times F \to \R$.\\

Next, in order to construct a compatible contact structure on $E$, we will use the above description as follows: By assumption $B$ admits a contact structure, say $\mu$. For any real number $K>0$, the $1$-form $K\mu$ is contact (indeed, the manifolds $(B,\textrm{Ker}(\mu))$ and $(B,\textrm{Ker}(K\mu))$ are contactomorphic, and their contact hyperplane distributions coincide everywhere on $B$). Also from Lemma \ref{lem:Contact_Str_on_M^Odd_cross_Exact} we know that for any $K>0$ the $1$-form $K\mu + \beta$ is a contact form on $B \times F$, and so, in particular, on each $W^i \times F$. (For simplicity, we will denote the pull-back of $\mu$ (resp. $\beta$) on any product space still by $\mu$ (resp. by $\beta$).\\

For a sufficiently small $\epsilon>0$ consider the tubular neighborhood $H_1 \times [-\epsilon,\epsilon]$ of $H_1=W^1 \cap W^2$ in $W^1 \cup W^2$, and let $f_1:H_1 \times [-\epsilon,\epsilon] \times F \longrightarrow \R$ be a smooth cut-off function defined by $f_1(b,x,p)=f(x)$ where $f:[-\epsilon,\epsilon] \to \R$ is the smooth cut-off function such that $f(x)\equiv 1$ on $ (-\delta,\delta)$ for some $0<\delta<\epsilon$ and $f(x)\equiv 0$ for all $x$ near $ \pm \epsilon$ as in Figure \ref{fig:Cut_off_function}. \\

\begin{figure}[ht]
\begin{center}
\includegraphics{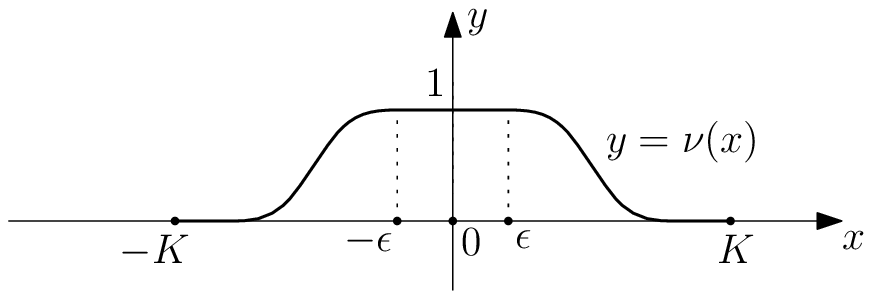}
\caption{Smooth cut-off funtion  $f(x)$}
\label{fig:Cut_off_function}
\end{center}
\end{figure}

For a real number $K_1>0$, consider the smooth $1$-form $\sigma_1:=K_1\mu + \beta +  f_1 d \Psi_1$ on $H_1 \times [-\epsilon,\epsilon] \times F$. By construction we can smoothly glue this form with the contact form $K_1\mu + \beta$ on $[(W^1 \cup W^2)-(H_1 \times (-\epsilon,\epsilon)] \times F$. Also observe that $\sigma_1$ is in accordance with the gluing map $\Phi^1$ because for any fiber $F$ over the gluing region $H_1$ (where $f_1\equiv 1$) we have $\sigma_1|_F=\beta + d\psi_1=\phi_1^*(\beta)=(\Phi^1|_F)^*(\beta)$. Therefore, $\sigma_1$ is, indeed, a smooth $1$-form on $(W^1 \times F) \cup_{\Phi^1} (W^2 \times F )$. Moreover, on the region $H_1 \times [-\epsilon,\epsilon] \times F$ we compute
\\
\begin{eqnarray} \label{eqn:contact_condition}
\sigma_1 \wedge (d\sigma_1)^{n+m}&=&C_1K_1^{n+1}\mu(d\mu)^{n}(d\beta)^{m}+C_2K_1^{n}f_1'\,\mu(d\mu)^{n-1}(d\beta)^{m}dx\,d \Psi_1\\
\nonumber &+&C_3K_1^{n}(d\mu)^{n}(d\beta)^{m}d \Psi_1 + C_4K_1^{n}f_1'\,(d\mu)^{n}\beta(d\beta)^{m-1}dx\,d \Psi_1
\end{eqnarray}
\\
where $C_1,C_2,C_3,C_4$ are constants detemined by the binomial expansion formula for $(d\sigma_1)^{n+m}$, and we use the facts $(d\mu)^i=0, \; \forall i\geq n+1$ and $(d\beta)^j=0, \; \forall j\geq m+1$. Observe that if $K_1$ is chosen large enough, the first term on the right hand side is dominant over all other terms since it has the highest degree of $K_1$. Therefore, for $K_1$ large enough $\sigma_1 \wedge (d\sigma_1)^{n+m}$ is a volume form on $H_1 \times [-\epsilon,\epsilon] \times F$, and hence $\sigma_1$ is a contact form on $(W^1 \times F) \cup_{\Phi^1} (W^2 \times F )$.\\

Now we want to extend $\sigma_1$ to a contact form on $(W^1 \times F) \cup_{\Phi^1} (W^2 \times F ) \cup_{\Phi^2} (W^3 \times F)$: For $\epsilon>0$ sufficiently small, consider the tubular neighborhood $H_2 \times [-\epsilon,\epsilon]$ of $H_2=(W^1 \cup W^2) \cap W^3$ in $W^1 \cup W^2 \cup W^3$ such that $H_2 \times [-\epsilon,0] \subset W^3$ and $H_2 \times [0,\epsilon] \subset W^1 \cup W^2$. Take a smooth cut-off function $g_2:H_2 \times [-\epsilon,\epsilon] \times F \longrightarrow \R$ defined by $g_2(b,x,p)=g(x)$ where $g:[-\epsilon,\epsilon] \to \R$ is the smooth cut-off function such that $g(x)\equiv 1$ on $ (-\delta,\epsilon]$ for some $0<\delta<\epsilon$ and $g(x)\equiv 0$ for all $x$ near $-\epsilon$ as in Figure \ref{fig:Cut_off_function_2}. \\

For a real number $K_2>0$, we define the smooth $1$-form $\sigma_2:=K_2\mu + \beta +  f_2 d \Psi_2$ on $H_2 \times [-\epsilon,\epsilon] \times F$ where $f_2=g_2f_1$. On the region $H_2 \times [-\epsilon,\epsilon] \times F$, similar to the computation given in Equation (\ref{eqn:contact_condition}) we have

\begin{figure}[ht]
\begin{center}
\includegraphics{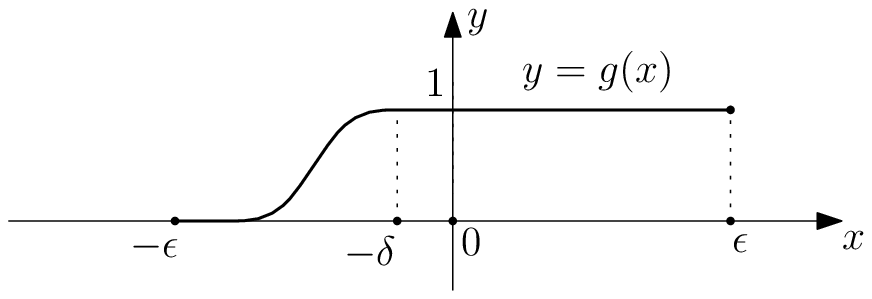}
\caption{Smooth cut-off funtion $g(x)$}
\label{fig:Cut_off_function_2}
\end{center}
\end{figure}

\begin{eqnarray} \label{eqn:contact_condition_2}
\sigma_2 \wedge (d\sigma_2)^{n+m}&=&C_1K_2^{n+1}\mu(d\mu)^{n}(d\beta)^{m}+C_2K_2^{n}f_2'\,\mu(d\mu)^{n-1}(d\beta)^{m}dx\,d \Psi_2\\
\nonumber &+&C_3K_2^{n}(d\mu)^{n}(d\beta)^{m}d \Psi_2 + C_4K_2^{n}f_2'\,(d\mu)^{n}\beta(d\beta)^{m-1}dx\,d \Psi_2
\end{eqnarray}

which implies that if $K_2$ is chosen large enough, then $\sigma_2$ is contact on $H_2 \times [-\epsilon,\epsilon] \times F$.
Moreover, the functions $\Psi_1$ and $\Psi_2$ are identical on their overlaping domains because $\phi_1$ and $\phi_2$ are constructed using the same transition functions on their common domains. As a result, if we choose $K_1=K_2=\textrm{Max}(K_1,K_2)$, then by gluing with $\sigma_1$ we obtain a smooth contact form $\sigma_2$ on $(W^1 \times F) \cup_{\Phi^1} (W^2 \times F ) \cup_{\Phi^2} (W^3 \times F)$.\\

Next, we can repeatedly follow exactly the same lines of the last extension argument for the remaining parts of the description of $E$ given in Equation (\ref{eqn:Description}), and eventually obtain a contact form on $E$. \\

More precisely, for each $i=3,4,...,r-1$ and $\epsilon>0$ small enough, consider the smooth cut-off function $g_i:H_i \times [-\epsilon,\epsilon] \times F \longrightarrow \R$ defined by $g_i(b,x,p)=g(x)$. Here, similar to above, we assume that $H_i \times [-\epsilon,0] \subset W^{i+1}$ and $H_i \times [0,\epsilon] \subset W^1 \cup \cdots \cup W^i$. Suppose that the contact form $$\sigma_{i-1}=K_{i-1}\mu + \beta +  f_{i-1} d \Psi_{i-1}, \quad \textrm{where} \quad f_{i-1}=g_{i-1}\cdots g_2f_1,$$ is already constructed on the union $(W^1 \times F) \cup_{\Phi^1} \cdots \cup_{\Phi^{i-1}} (W^{i} \times F )$, and it extends $\sigma_{i-2}$. Note that one needs to choose the constants $K_j$ ($j=1,2,...,i-1$) to be $K_1=K_2=\cdots = K_{i-1}=\textrm{Max}(K_1,K_2,...,K_{i-1})$ to perform this extension. Then one can similarly check as in Equation (\ref{eqn:contact_condition_2}) that the form $$\sigma_i:=K_i\mu + \beta +  f_i d \Psi_i, \quad \textrm{where} \quad f_{i}=g_i f_{i-1}=g_ig_{i-1}\cdots g_2f_1,$$ is contact on $H_i \times [-\epsilon,\epsilon] \times F$ if $K_i$ is large enough. Moreover, one can see $\sigma_i$ as the extension of $\sigma_{i-1}$ to the union $(W^1 \times F) \cup_{\Phi^1} \cdots \cup_{\Phi^{i}} (W^{i+1} \times F )$ once we do the reassignments $$K_1=K_2=\cdots = K_{i}=\textrm{Max}(K_1,K_2,...,K_{i}).$$

Finally, we set $\sigma:=\sigma_{r-1}$. Then by the construction $\sigma$ restricts to an exact symplectic structure on each fiber of $\pi$. So the contact form $\sigma$ defines a contact structure $\textrm{Ker}(\sigma)$ on $E$ which is compatible with the fibration map $\pi:E \to B$.\\

In order to prove the second statement, observe that if $G\subset \textrm{Exact}(F,\partial F,\beta)$, then for any transition map $\phi_{\alpha \gamma}$ used in the above construction and for any $b \in  U_\alpha \cap U_\gamma$ we have $\phi_{\alpha \gamma}(b)\equiv Id_F$ near $\partial F$ (in particular, $\partial_hE \approx B \times \partial F$). Therefore, the smooth maps $\psi_i:F \to \R$ in the construction are all constant functions near $\partial F$, and so the corresponding functions $\Psi_i:H_i \times F \to \R$ take constant values near $\partial_hE$ (i.e., $d\Psi_i=0$ near $\partial_hE$). As a result, the description of the contact form $\sigma(=\sigma_{r-1})$ on $E$ given above takes the form $\sigma=K_{r-1}\mu + \beta$ near $\partial_hE$.
\end{proof}

%==============================================================================
%
%         Bundle Contact Structures
%
%==============================================================================

\section{Bundle Contact Structures} \label{sec:Bundle_Contact_Structures}
From the previous section any contact symplectic fibration admits a (co-oriented) contact structure. Here we study the flexibility of such contact structures. More specifically, we are interested in how they react if we change the contact structure on the base in its isotopy class. Before we proceed let us first fix some terminonology and notation.

\begin{definition}
Let $(F\hookrightarrow E \stackrel{\pi}{\twoheadrightarrow} B)_G$ be a contact symplectic fibration with the contact base $(B,\textrm{Ker}(\mu))$ and the structure group $G \subset \textrm{Exact}(F,\beta)$. The contact structure $\textrm{Ker}(\sigma)$ on $E$ which is defined by the contact form $$\sigma=\sigma_{r-1}=K_{r-1}\mu + \beta +  f_{r-1} d \Psi_{r-1},$$ constructed as in the proof of Theorem \ref{thm:Existence_on_contact_symp_fib}, is called a \emph{bundle contact structure for} $\pi$ \emph{associated to} $\mu$, and will be denoted by $\xi(\mu)$.
\end{definition}

Note by definition bundle contact structures are compatible with the fibration maps. Next we show that isotopic contact structures on the base space produce isotopic bundle contact structures on the total space.

\begin{theorem} \label{thm:isotopic_to_isotopic}
Let $(F\hookrightarrow E \stackrel{\pi}{\twoheadrightarrow} B)_G$ be a contact symplectic fibration. Suppose co-oriented contact structures $\eta_i=\emph{Ker}(\mu_i)$, $i=0,1$, on $B$ are isotopic to each other. Then the corresponding bundle contact structures $\xi(\mu_0), \xi(\mu_1)$ on $E$ are isotopic through bundle contact structures on $E$. In particular, $(E,\xi(\mu_0))$ and $(E,\xi(\mu_1))$ are contactomorphic.
\end{theorem}

\begin{proof}
Since $\eta_0,\eta_1$ are isotopic, there exists a smooth family $\{\eta_t \,| \, t \in [0,1]\}$ of contact structures on $B$ joining them. Consider a smooth family $\{\alpha_t \,| \, t \in [0,1]\}$ of contact forms on $B$ such that $\alpha_0=\mu_0$ and $\textrm{Ker}(\alpha_t)=\eta_t$ for all $t$. Since $\textrm{Ker}(\alpha_1)=\eta_1=\textrm{Ker}(\mu_1)$, we have $\alpha_1=h\mu_1$ for some smooth function $h:B \to \R^+$. Then we obtain the (modified) smooth family $$\{\mu_t \,| \, t \in [0,1]\}, \quad \mu_t:=\left(1-t+\dfrac{t}{h}\right)\alpha_t$$ of contact forms on $B$ joining $\mu_0$ and $\mu_1$ and with the property $\textrm{Ker}(\mu_t)=\eta_t$ for all $t$. \\

Consider the bundle contact structure $\xi(\mu_t)$ on $E$ associated to $\mu_t$ given by the contact form $\sigma_t=K^t_{r-1}\mu_t + \beta +  f_{r-1} d \Psi_{r-1}$ constructed as in the proof of Theorem \ref{thm:Existence_on_contact_symp_fib}. In order to simplify the notation, we will write $K_t$ for $K^t_{r-1}$. Let us consider
\begin{center}
$\mathcal{K}:=\{K_t \, | \, t \in [0,1] \}$.
\end{center}
For each $t$ the choice of $K_t$ can be done in such a way that $K_t$'s varies continuously with respect to $t$. This is possible since $\{\mu_t \,| \, t \in [0,1]\}$ is a smooth $1$-parameter family. Therefore, the set $\mathcal{K}$ can be realized as the range set of a continuous function on the compact set $[0,1]$, and hence there exists an upper bound, say $K_{upp}$, for $\mathcal{K}$.
We set $ K=\textrm{Max} (K_0,K_{upp})$, and consider three smooth $1$-parameter families:

\begin{eqnarray}
\nonumber \lambda^1_t &:=& [(1-t)K_0+tK]\mu_0 + \beta +  f_{r-1} d \Psi_{r-1}, \quad t \in [0,1], \\
\nonumber \lambda^2_t &:=& K\mu_t + \beta +  f_{r-1} d \Psi_{r-1}, \quad t \in [0,1],\\
\nonumber \lambda^3_t &:=& [(1-t)K+tK_1]\mu_1 + \beta +  f_{r-1} d \Psi_{r-1}, \quad t \in [0,1].
\end{eqnarray}

As $[(1-t)K_0+tK]\geq K_0$ (resp. $K\geq K_t$ and $[(1-t)K+tK_1]\geq K_1$) for each $t$, the construction in Theorem \ref{thm:Existence_on_contact_symp_fib} implies that $\lambda^1_t$ (resp. $\lambda^2_t$ and $\lambda^3_t$) is a contact form on $E$ for any $t \in [0,1]$. Now we define a concatenation by the rule:
$$ \Lambda_t:=
\begin{cases}
\lambda^1_{3t} & \textrm{if} \quad t \in [0,1/3] \\
\lambda^2_{3t-1} & \textrm{if} \quad t \in [1/3,2/3]\\
\lambda^3_{3t-2} & \textrm{if} \quad t \in [2/3,1]
\end{cases}.
$$
Observe that $\Lambda_0=\sigma_0$ and $\Lambda_1=\sigma_1$, and so $\{\Lambda_t \, | \, t \in [0,1]\}$ is a smooth $1$-parameter family of contact forms on $E$ joining $\sigma_0$ and $\sigma_1$. As a result, we obtain a smooth $1$-parameter family $$\{\xi_t:=\textrm{Ker}(\Lambda_t) \, | \, t \in [0,1]\}$$ of (bundle) contact structures on $E$ joining $\xi_0=\textrm{Ker}(\sigma_0)=\xi(\mu_0)$ and $\xi_1=\textrm{Ker}(\sigma_1)=\xi(\mu_1)$. Hence, $\xi(\mu_0)$ and $\xi(\mu_1)$ are isotopic by Gray's Stability Theorem (see, for instance, Theorem 2.2.2 of \cite{Ge}).
\end{proof}

%==============================================================================
%
%         Fiber Connected Sum
%
%==============================================================================

\section{Fiber Connected Sum} \label{sec:Fiber_connected_sum}
In this section we prove that the fiber connected summing of two contact symplectic fibrations along their fibers results in another contact symplectic fibration, and, in particular, the sum admits a (compatible) bundle contact structure and it fibers over the contact connected sum of the original contact bases. We first recall the fiber connected sum operation in a slightly more general situation than the ones described in \cite{Geiges,Ge}. Namely, we also allow the ambient manifolds to be noncompact or compact with boundary.

\begin{definition} \label{def:Fiber_Connected_Sum}
For an oriented manifolds $M_1, M_2$, let $\phi_i: \Sigma \hookrightarrow M_i$, be two codimension$-k$ embeddings of an oriented manifold $\Sigma$, and  for each $i$ denote by $N_i$ the normal bundle of $\phi_i(\Sigma)$ in $M_i$ (which is an oriented $D^k-$bundle over $\phi_i(\Sigma)$). Let $\Phi:N_1 \to N_2$ be a fiber-orientation-reversing bundle isomorphism covering the composition $\phi_2 \circ \phi_1^{-1}|_{\phi_1(\Sigma)}$. After picking a Riemannian metric and normalization, one can assume that $\Phi$ is norm-preserving. For an interval $(a,b)$, where $0\leq a<b$, define the subset $N_i^{(a,b)}:=\{(p,x) \in N_i \, | \, a< \Vert x \Vert <b\}$. Furthermore, for a given $\epsilon>0$ consider the orientation-preserving diffeomorphism (or identification)
\begin{eqnarray*}
N_1^{(\epsilon/2,\sqrt{3}\epsilon/2)} & \stackrel{\sim_{\Phi}}{\longleftrightarrow} &N_2^{(\epsilon/2,\sqrt{3}\epsilon/2)}
\end{eqnarray*}
given explicitly by the rule
\begin{eqnarray} \label{eqn:Identification}
(p,x) &  \stackrel{\sim_{\Phi}}{\longleftrightarrow} & \left( \phi_2 \circ \phi_1^{-1}(p), \dfrac{\sqrt{\epsilon^2-\Vert x \Vert^2}}{\Vert x \Vert} \cdot \Phi_F(x)  \right)
\end{eqnarray}
where $\Phi_F:D^k \to D^k$ denotes the fiber component of $\Phi$. Then the \emph{fiber connected sum} $M_1 \#_{\Phi} M_2$ of $M_1$, $M_2$ \emph{along} their submanifolds $\phi_1(\Sigma),\phi_2(\Sigma)$ is the oriented manifold  given as the quotient space
\begin{center}
$M_1 \#_{\Phi} M_2:=[(M_1 - N_1^{[0,\epsilon/2]}) \bigcup (M_2 - N_2^{[0,\epsilon/2]})] \;/ \sim_{\Phi}$.
\end{center}
\end{definition}

We remark that in the special case when $M_1=M=M_2$, in order to get a smooth manifold one should also assume that the embeddings $\phi_1,\phi_2$ are disjoint. Now we are ready to prove

\begin{theorem} \label{thm:Fiber_conn_sum_along_convex_hyper}
Let $(F\hookrightarrow E_j \stackrel{\pi_j}{\twoheadrightarrow} B_j)_{G_j}$ be two contact symplectic fibrations with $\emph{dim}(B_1)=\emph{dim}(B_2)=2n+1$. For an exact symplectic manifold $(F,\beta)$, assume that there exist embeddings $\phi_1:F\hookrightarrow E_1$, $\phi_2:F\hookrightarrow E_2$ of the fibers $F_1=\phi_1(F)$ and $F_2=\phi_2(F)$  (which are distinct if $E_1=E_2$ and $\pi_1=\pi_2$) of $\pi_1$ and $\pi_2$, respectively, and for each $j=1,2$ the total space $E_j$
%and its bundle contact structure $\xi_{\sigma_j}=\emph{Ker}(\sigma_j)$ are
is constructed by regarding $G_j \subset \emph{Exact}(F,\beta)$.
Then if $n$ is odd (resp. even), then one can form the fiber connected sum $E$ of $E_1$ and $E_2$ (resp. $E_1$ and $\overline{E_2}$) along $F_j$'s using some canonical fiber-orientation-reversing bundle isomorhism $\Phi:N_1 \to N_2$ between normal bundles of $F_j$'s such that $E=E_1 \#_{\Phi} E_2$ (resp. $E=E_1 \#_{\Phi} \overline{E_2}$) admits a contact structure $\xi$ which restricts to a bundle contact structure on each piece $E-E_j$ (resp. on $E-E_1$ and $E-\overline{E_2}$). Moreover, there exists a contact symplectic fibration $$(F\hookrightarrow E \stackrel{\pi}{\twoheadrightarrow} B)_G,$$ where $B=B_1 \# B_2$ if $n$ is odd and $B=B_1 \# \overline{B_2}$ if $n$ is even, such that $B_1 \# B_2$ (resp. $B_1 \# \overline{B_2}$) is the contact connected sum of the original contact bases $B_1,B_2$ (resp. $B_1, \overline{B_2}$), $G \subset \emph{Exact}(F,\beta)$, $\pi=\pi_j$ on their overlaping domains, and $\xi$ is a bundle contact structure for $\pi$.
\end{theorem}

\begin{proof} First note that each $B_j$ admits a contact form $\mu_j$ from which we can construct a contact form $\sigma_j$ on $E_j$ using the pull-back form $\pi_j^*(K_j\mu_j)$ by taking $K_j>0$  large enough. (See the proof of Theorem \ref{thm:Existence_on_contact_symp_fib}). Also in the case when $n$ is even, since the orientation on $\overline{B_2}$ is opposite of the one on $B_2$, one should consider the contact form $-\mu_2$ in order to define a positive contact structure on $\overline{B_2}$.  By assumption both $E_1$ and $E_2$ can be equipped with a fiber bundle structure such that structure groups $G_1,G_2 \subset \textrm{Exact}(F,\beta)$. Consider the decomposition of each $E_j$ (as in Equation (\ref{eqn:Description})) given by
$$E_j=(W_j^1 \times F) \cup_{\Phi_j^1} (W_j^2 \times F ) \cup_{\Phi_j^2} \cdots \cup_{\Phi_j^{r_j-1}} (W_i^{r_j} \times F )$$
where we consider the decomposition of each $B_j$ (as in Equation (\ref{eqn:Partition})) given by $$B_j=W_j^1 \cup W_j^2 \cup \cdots \cup W_j^{r_j}.$$
From differential topology point of view we are free to choose the fibers $F_1$ and $F_2$ along which the fiber connected sum is performed because different choices of fibers produces diffeomorphic total spaces which fibers over diffeomorphic bases. Therefore, by changing the fiber $F_i$ (if necessary), we may assume that the embedding $\phi_j$ maps $F$ onto a fiber $F_j=\phi_j(F)$ which lies in $\textrm{int}(W_j^l) \times F$ for some $l\in \{1,...,r_j\}$. Assume that each $F_j$ is the fiber over the point $b_j \in \textrm{int}(W_j^l) \subset B_j$, that is, $F_j=\pi_j^{-1}(b_j)$. \\

%Let $K>0$ be any real number.
By Darboux theorem (see, for instance, Theorem 2.5.1 of \cite{Ge}), for each $j$ there exist a neighborhood $U_j\cong \D^{2n+1}$ of $b_j$ which is contained in $\textrm{int}(W_j^l)$ and coordinates $(\textbf{u},\textbf{v},w)$ (where $\textbf{u}=(u_1,...,u_n)$ and $\textbf{v}=(v_1,...,v_n$)) on $U_j$ such that $$K_j\mu_j=dw+\textbf{u}d\textbf{v} \qquad \textrm{on} \quad U_j.$$
By a further change of coordinates on $\D^{2n+1}\subset \R^{2n+1}$ (see Example 2.1.3 of \cite{Ge} for details) given by $(\textbf{u},\textbf{v},w)\to (\textbf{x}=(\textbf{u} + \textbf{v})/2, \textbf{y}=(\textbf{v} - \textbf{u})/2, z=w + \textbf{u}\cdot \textbf{v}/2)$, for each $j$ we get the local description
\begin{equation} \label{eqn:Local_description}
K_j\mu_j=dz+\sum_{k=1}^{n}(x_kdy_k-y_kdx_k)=dz+\sum_{k=1}^{n}r_k^2d\theta_k
\end{equation}
on $U_j$. Here $\textbf{x}=(x_1,...,x_n)$, $\textbf{y}=(y_1,...,y_n$), and $(r_k,\theta_k)$ are the polar coordinates on the $(x_k,y_k)$-plane.\\

Now consider the tubular neighborhood $\pi_j^{-1}(U_j)=U_j \times F \subset \textrm{int}(W_j^l) \times F$ of $F_j$ in $E_j$, denote by $N_j$ the normal bundle of $F_j$ in $E_j$, and let $\mathcal{P}_j:N_j \to F_j$ be the bundle projection. The tubular neighborhood theorem implies that (by taking a smaller $U_j$ if necessary) we can write $N_j=F \times U_j$ and $F_j=F \times \{\textbf{0}\}$. Then from the construction of $\sigma_j$, by taking $U_j$ small enough, one can guarantee that it is located far enough from $\partial W_j^l$, and so $\sigma_j=\pi_j^*(K_j\mu_j) + \mathcal{P}_j^*(\beta)$ on $F \times U_j$. Combining this with Equation (\ref{eqn:Local_description}), we obtain the local description
\begin{equation} \label{eqn:contact_form_locally}
\nonumber \sigma_j=dz+\sum_{k=1}^{n}r_k^2d\theta_k +\beta \qquad \textrm{on} \quad N_j=F \times U_j.
\end{equation}
where for simplicity (as in the proof of Theorem \ref{thm:Existence_on_contact_symp_fib}) we still write $dz+\sum_{k=1}^{n}r_k^2d\theta_k$ (resp. $\beta$) for the pull-back form $\pi_j^*(dz+\sum_{k=1}^{n}r_k^2d\theta_k)$ (resp. $\mathcal{P}_j^*(\beta)$) on $F \times U_j$.
We will form the fiber connected sum of $E_1$, $E_2$ along $F_1,F_2$ using the (canonical) bundle isomorhism $\Phi:N_1 \to N_2$ of the normal bundles $N_j=F_j \times U_j \subset F_j \times \R^{2n+1}$ which is given by
\begin{center}
$\Phi(p,z,r_1,\theta_1, ...,r_n,\theta_n)=
\begin{cases}
 (\phi_2 \circ \phi_1^{-1}(p),z,-r_1,\theta_1, ...,-r_n,\theta_n)& \textrm{if} \;\; $n$ \; \textrm{is odd}, \\
 (\phi_2 \circ \phi_1^{-1}(p),-z,r_1,-\theta_1, ...,r_n,-\theta_n)& \textrm{if} \;\; $n$ \; \textrm{is even}.
\end{cases}$
\end{center}
Observe that in both cases $\Phi$ is fiber-orientation-reversing, and it covers the composition $\phi_2 \circ \phi_1^{-1}|_{F_1}$. Moreover, $r^2=z^2+r_1^2+\cdots r_n^2$ where $r$ denotes the radial coordinate on $U_j\cong \D^{2n+1}$, and so with respect to the standard metric on $\R^{2n+1}$, $\Phi$ is norm preserving. In particular, for $\epsilon>0$ chosen sufficiently small, $\Phi$ is a bijection between the subsets $$N_j^{[0,\epsilon)}=\{(p,x) \in N_j \, | \, 0\leq \Vert x \Vert \leq \epsilon\} \subset F \times U_j=N_j, \quad j=1,2.$$
Let $\Upsilon:N_1^{(\epsilon/2,\sqrt{3}\epsilon/2)}\to N_2^{(\epsilon/2,\sqrt{3}\epsilon/2)}$ be the orientation-preserving diffeomorphism (obtained from $\Phi$) corresponding to the identification (given in (\ref{eqn:Identification})) constructing $E:=E_1 \#_{\Phi} E_2$. That is, $$\Upsilon(p,x)=\left(p, \dfrac{\sqrt{\epsilon^2-\Vert x \Vert^2}}{\Vert x \Vert} \cdot \Phi_F(x)  \right)$$ where
\begin{center}
$\Phi_F(z,r_1,\theta_1, ...,r_n,\theta_n)=
\begin{cases}
 (z,-r_1,\theta_1, ...,-r_n,\theta_n)& \textrm{if} \;\; $n$ \; \textrm{is odd}, \\
 (-z,r_1,-\theta_1, ...,r_n,-\theta_n)& \textrm{if} \;\; $n$ \; \textrm{is even}.
\end{cases}$
\end{center}
 is the fiber component of $\Phi$. Consider the function $$h:(\epsilon/2,\sqrt{3}\epsilon/2) \to (\epsilon/2,\sqrt{3}\epsilon/2), \quad h(r)=\sqrt{\epsilon^2-r^2}, \quad (r=\Vert x \Vert)$$ which is used in defining the gluing map $\Upsilon$. Since $r=\epsilon/\sqrt{2}$ is the (only) fixed point of $h$, $\Upsilon$ identifies the trivial $2n$-sphere bundles $S_j:=F\times H_j \subset N_j$ where $$H_j:= \{x \in U_j=\D^{2n+1} \, | \, r= \Vert x \Vert =\epsilon/\sqrt{2}\} \cong \S^{2n}$$ is the $2n$-sphere in $U_j=\D^{2n+1}$ of radius $r=\epsilon/\sqrt{2}$. For each $j$, $N_j^{(\epsilon/2,\sqrt{3}\epsilon/2)}$ deformation retracts onto $S_j$, and so up to diffeomorphism we may regard $E$ as
 \begin{center}
$E=E_1 \#_{\Phi} E_2\cong[(E_1 - N_1^{[0,\epsilon/\sqrt{2})}) \bigcup (E_2 - N_2^{[0,\epsilon/\sqrt{2})})] \;/ \sim$
\end{center}
where $\sim$ identifies the boundaries $S_j=E_j - N_j^{[0,\epsilon/\sqrt{2})}$ via the restriction map $$\Upsilon|_{S_1}:S_1 \to S_2.$$
On $S_j=F\times H_j$, we compute
\begin{eqnarray*} \label{eqn:Final_computation}
\nonumber (\Upsilon|_{S_1})^*(\sigma_2) &=&(\Upsilon|_{S_1})^*\left(dz+\sum_{k=1}^{n}r_k^2d\theta_k   +  \beta\right)
=\Phi_F^*\left(dz+\sum_{k=1}^{n}r_k^2d\theta_k\right)   + \beta \\
&=&
\begin{cases}
 \quad \; dz+\sum_{k=1}^{n}r_k^2d\theta_k   +  \beta & \textrm{if} \;\; $n$ \; \textrm{is odd}, \\
 -\left(dz+\sum_{k=1}^{n}r_k^2d\theta_k\right)   +  \beta& \textrm{if} \;\; $n$ \; \textrm{is even}.
\end{cases}
\end{eqnarray*}
This shows that if $n$ is odd, then $(\Upsilon|_{S_1})^*(\sigma_2)=\sigma_1$, i.e., $\Upsilon|_{S_1}$ is a contactomorphism, and so by gluing $\sigma_j$'s along $S_j$'s we obtain a contact form, say $\sigma$, on $E=E_1\#_\Phi E_2$. If $n$ is even, the above computation suggests that by switching the orientations on $B_2$ (and so on $E_2$) and assuming the (positive) contact form $\sigma_2$ on $\overline{E_2}$ is constructed using the positive contact form $-K_2\mu_2$ on $\overline{B_2}$, we can make $\Upsilon|_{S_1}$ a contactomorphism (because we will have $(\Upsilon|_{S_1})^*(\sigma_2)=\sigma_1$ for this case as well), and therefore, one can glue $\sigma_j$'s along $S_j$'s and obtain a contact form $\sigma$ on $E=E_1\#_\Phi \overline{E_2}$. \\

By construction the contact structure $\xi:=\textrm{Ker}(\sigma)$ on $E$ restricts to the original bundle contact structure on each piece $E-E_j$ (if $n$ is odd), or on each piece $E-E_1$ and $E-\overline{E_2}$ (if $n$ is even). Moreover, in both cases during the gluing process (along $S_j$'s) only the base coordinates (that is, the coordinates on $H_j$'s) are nontrivially identified, and so one can smoothly glue the maps $\pi_1,\pi_2$ together, and obtain a fiber bundle projection $\pi$ on $E$ with fibers $F$. On the other hand, the gluing map $\Upsilon|_{S_1}$ is identity on each fiber $F$ over $H_1$, and, therefore, the structure group $G$ of $\pi$ is contained in $\textrm{Exact}(F,\beta)$. \\

Next observe that at the base level, $\Upsilon|_{H_1}$ pulls back $dz+\sum_{k=1}^{n}r_k^2d\theta_k$ to itself (if $n$ is odd), and to $-(dz+\sum_{k=1}^{n}r_k^2d\theta_k)$ (if $n$ is even). Thus, constructing the total space $E$ upstairs corresponds to taking the contact connected sum of $(B_j,\mu_j)$'s (if $n$ is odd), or of $(B_1,\mu_1)$ with $(\overline{B_2},-\mu_2)$ (if $n$ is even) downstairs as claimed.\\

Finally, consider the reassignments $K_1=K_2=K$ where $K:=\textrm{Max}(K_1,K_2)$. Then using these new choices of $K_1,K_2$ we can carry out all the constructions above in the same way, and corresponding conclusion are still be correct. For instance, in the case that $n$ is odd, one can see for each $i$ that $K_i\mu_i$ and $K\mu_i$ define isotopic (constant isotopy) contact structures on $B_i$, and so, Theorem \ref{thm:isotopic_to_isotopic}, the bundle contact structures $\xi(K_i\mu_i),\xi(K\mu_i)$ on $E_i$ are isotopic. Hence, the corresponding contact structures $\xi$'s (kernels of $\sigma$'s) on $E=E_1\#_\Phi E_2$ constructed above are isotopic.\\

Now it is straight forward to check that when $n$ is odd (resp. $n$ is even) the contact structure $\xi=\textrm{Ker}(\sigma)$ on $E=E_1\#_\Phi E_2$ (resp. $E=E_1\#_\Phi \overline{E_2}$) is the bundle contact structure for $\pi$ associated to $\mu_1 \cup_{\Upsilon|_{H_1}} \mu_2$ (resp. $\mu_1 \cup_{\Upsilon|_{H_1}} -\mu_2$)
\end{proof}

%==============================================================================
%==============================================================================
%==============================================================================

%\clearpage

\end{document}